\documentclass[12pt,reqno]{amsart}
\usepackage{,amsmath,amsthm,amssymb,euscript}%
\newtheorem{theorem}{Theorem}
\newtheorem{prop}[theorem]{Proposition}
\newtheorem{lemma}[theorem]{Lemma}
\theoremstyle{definition}
\newtheorem*{remark}{Remark}

\newcommand{\F}{{\EuScript F}}
\newcommand{\G}{{\EuScript G}}
\renewcommand{\H}{{\EuScript H}}
\newcommand{\V}{\EuScript V}
\newcommand{\PP}{\mathbb P}
\renewcommand{\P}{\mathbb P}
\newcommand{\R}{\mathbb R}
\newcommand{\C}{\mathbb C}
\newcommand{\CP}{\C\PP}
\newcommand{\CH}{\C{\mathrm H}}

\newcommand{\CC}{{\EuScript C}}
\newcommand{\I}{{\mathcal I}}
\newcommand{\K}{{\mathcal K}}
\newcommand{\Khat}{\widehat{\K}}
\newcommand{\M}{{\mathcal M}}

\newcommand{\w}{\omega}
\newcommand{\khat}{\hat\kappa}

\newcommand{\fhat}{\hat{f}}

\newcommand{\setU}{U}
\newcommand{\Uhat}{\widehat{\setU}}
\newcommand{\Bbar}{\overline{B}}
\newcommand{\pbar}{\bar{p}}
\newcommand{\qbar}{\bar{q}}

\newcommand{\bv}{\mathbf v}
\newcommand{\bg}{\mathbf g}
\newcommand{\bn}{\mathbf n}
\newcommand{\bz}{\mathbf z}

\def\vz{{\sf z}}
\def\vw{{\sf w}}
\def\e{{\sf e}}
\def\ehat{\hat{\e}}
\def\emp{\epsilon}

\newcommand{\ri}{\mathrm i}
\newcommand{\realpart}{\operatorname{Re}}
\newcommand{\impart}{\operatorname{Im}}

\newcommand{\JJ}{{\mathrm J}} 
\newcommand{\di}{\partial}

\newcommand{\restr}{\negthickspace \mid}

\newcommand{\sfP}{{\sf P}}

\newcommand{\intprod}{\mathbin{\raisebox{.4ex}{\hbox{\vrule height .5pt width
5pt depth 0pt %
         \vrule height 3pt width .5pt depth 0pt}}}}

\renewcommand\&{\wedge}

\newcommand{\jbar}{{\overline{\jmath}}}

\begin{document}
\title{A d'Alembert Formula for Hopf Hypersurfaces}
\author{Thomas A. Ivey}
\date{\today}
\begin{abstract}
A Hopf hypersurface in complex hyperbolic space $\CH^n$ is one
for which the complex structure applied to the normal vector is a principal
direction at each point.  In this paper, Hopf hypersurfaces for which the
corresponding principal curvature
is small (relative to ambient curvature) are studied by means of a generalized Gauss map into
a product of spheres, and it is shown that the hypersurface
may be recovered from the image of this map, via an explicit
parametrization.
\end{abstract}
\maketitle

\section*{Introduction}
A Hopf hypersurface is a real codimension-one submanifold $M$ in a complex
space form $\CP^n$ or $\CH^n$ which has the property that the
{\em structure directions} are principal directions at each point.
For any hypersurface in these spaces, a tangent vector $W$ is a structure
direction if $\JJ W$ is normal to the hypersurface, where $\JJ$ is
the complex structure.  If we assume that $M$ is oriented,
then we may globally define a structure vector field $W$ on
$M$ by requiring that $\JJ W$ be the unit normal vector field.

On a Hopf hypersurface, the principal curvature $\alpha$ of the structure direction
must be locally constant.  (This is due to Maeda \cite{maeda} for $\CP^n$ and
Ki-Suh \cite{KiSuh} for $\CH^n$.)
Furthermore, Cecil and Ryan \cite{cecilryan}
showed that every Hopf hypersurface in $\CP^n$ is locally congruent to a tube
over a holomorphic submanifold.  The radius $R$ of the tube is related to $\alpha$
by the formula $\alpha = (2/r)\cot(2R/r)$, where $4/r^2$ is the ambient holomorphic
sectional curvature.
Similarly, Montiel \cite{montiel} showed that Hopf hypersurfaces in $\CH^n$ (with
holomorphic sectional curvature $-4/r^2$) are tubes of radius $R$ over
holomorphic submanifolds, with $\alpha = (2/r)\coth(2R/r)$.  Notice
that $|\alpha| >2/r$ for these tubes;
a general construction for
Hopf hypersurfaces in $\CH^n$ with `small' $\alpha$ (i.e., $|\alpha|<2/r$) has been lacking, until now.

The Hopf condition for hypersurfaces in $\CH^n$, when regarded as a system of PDE,
changes type depending on $\alpha$:  it is elliptic
for `large' $\alpha$ (i.e., $|\alpha|>2/r$) and hyperbolic for small $\alpha$.
(In the elliptic case, the system is equivalent to the Cauchy-Riemann
equations for the holomorphic submanifolds which generate Montiel's tubes.)
The title of this paper is inspired by the fact that, like the standard
wave equation $u_{xx}=u_{tt}$, the system for Hopf hypersurfaces with small $\alpha$ is
integrable by the {\em method of Darboux} (see, e.g., Chapter 6 in \cite{cfb}).
For the wave equation, this integrability leads directly to d'Alembert's formula
$u(x,t)=f(x+t)+g(x-t)$ for the general solution, while for Hopf hypersurfaces we can
similarly construct the general solution in terms of arbitrary functional data.
Indeed, a similar approach has been recently taken by
Aledo, G\'alvez and Mira \cite{AGM} for flat surfaces in $S^3$.

In an earlier paper \cite{IR1}, Ryan and I showed that a Hopf hypersurface with small
$\alpha$ in $\CH^2$ determines, via a kind of generalized Gauss map, a pair of
contact curves in $S^3$.  We also showed that, conversely, the hypersurface may be constructed
(without integration) from an arbitrary pair of contact curves.  The main result of
this paper (see Theorem \ref{mainthm} below) is a straightforward generalization
of this construction to Hopf hypersurfaces in $\CH^n$, with
the contact curves being replaced by Legendrian submanifolds in $S^{2n-1}$.
(In all cases, the contact structure on the sphere is the standard one, with contact
planes perpendicular to the fibers of the Hopf fibration $S^{2n-1} \to \CP^{n-1}$.)
We also parametrize the Hopf hypersurfaces in $\CH^n$ arising from
this construction in terms of two given Legendrian submanifolds.  It is this
parametrization \eqref{dalembert} that we refer to
as a d'Alembert formula for Hopf hypersurfaces.

The main vehicle for producing results is the technique of moving frames,
and the use of exterior differential systems defined on various frame bundles.
We put this machinery in place by discussing moving
frames in complex space forms in \S1, and examining the exterior differential system for Hopf hypersurfaces in \S2.  The space $\CH^n$ inherits its geometry as the quotient
of the anti-de Sitter space $Q \subset \C^{n+1}$ by a circle action,
and the relationship between moving frames on $\CH^n$ and those on $Q$
is laid out in \S3.  This is used to prove the main result, and
also utilized in \S4 to obtain the d'Alembert formula.

\section{Moving Frames for Complex Space Forms}
Let $X$ be a simply-connected complex space form of dimension $n$, with constant holomorphic sectional curvature $4c$.
Let $\F$ be the unitary frame bundle of $X$, with fibration $\rho:\F\to X$.  Sections of this bundle are orthonormal frames
$e_1, e_2, \ldots, e_{2n}$ which satisfy the relations
$$\JJ e_{2k-1} = e_{2k}, \qquad 1\le k\le n.$$
Thus, with respect to this frame, the complex structure on $X$ is represented by the
matrix
\begin{equation}\label{defofJmatrix}
J = \begin{bmatrix} \begin{smallmatrix} 0 & -1 \\ 1 & 0 \end{smallmatrix} & & & \\
& \begin{smallmatrix} 0 & -1 \\ 1 & 0 \end{smallmatrix} & & \\
& & \ddots & \end{bmatrix}.
\end{equation}
The bundle $\F$ inherits, via pullback from the full orthogonal frame bundle,
canonical forms $\w^i$ and connection forms $\w^i_j$, where $1\le i,j \le 2n$.
These satisfy the structure equations
\begin{align}
d\w^i &= -\w^i_j \& \w^j,\\
d\w^i_j &= -\w^i_k \& \w^k_j + \Phi^i_j, \label{defofCurvature}
\end{align}
with curvature 2-forms given by
\begin{equation}\label{curvaturevalues}
\Phi^i_j = c(J^i_k \w^k \& J^j_\ell \w^\ell + \w^i \& \w^j - 2 J^i_j \Omega),
\end{equation}
where
$$\Omega := \w^1 \& \w^2 + \w^3 \& \w^4 + \ldots +\w^{2n-1} \& \w^{2n}$$
and we use summation convention for repeated indices from now on.
(As explained in \S\ref{Qstructure} below, the curvature may be
calculated using the fact that $X$ is a quotient of a quadric $Q \subset \C^{n+1}$
by an isometric circle action.)
Regarded as components of matrices of size $2n$,
the connection and curvature forms take value in $\mathfrak u(n) \subset \mathfrak{so}(2n)$,
which may be characterized as the subalgebra of matrices commuting with $J$.  In particular,
we note that
\begin{equation}\label{connrel}
\w^i_k J^k_j = J^i_k \w^k_j.
\end{equation}

Suppose that $(e_1, \ldots, e_{2n})$ is a unitary framing defined along a real hypersurface
$M \subset X$, adapted so that $e_{2n}$ is the hypersurface normal, and thus
$e_{2n-1}$ equals the structure vector $W$.  This framing defines a section
$f:M \to \F$ that satisfies $f^* \w^{2n}=0$ and
$$f^* \w^{2n}_j = A_{jk} \w^k, \qquad 1\le j,k \le 2n-1,$$
where matrix $A$ represents the shape operator of $M$
with respect to the basis for $TM$ provided by $(e_1, \ldots, e_{2n-1})$.
In what follows, we will identify the framing with the corresponding submanifold $f(M)\subset \F$.

In particular, an adapted framing for a Hopf hypersurface with principal
curvature equal to a given constant $\alpha$ will be an integral submanifold of the differential forms
\begin{align*}\theta_0 &:= \w^{2n},\\
\theta_1 &:= \w^{2n}_{2n-1} -\alpha \w^{2n-1}.
\end{align*}
(In other words, these differential forms pull back to vanish on the submanifold $f(M)$.)
Conversely, any $(2n-1)$-dimensional submanifold $\Sigma \subset \F$ that is an integral of $\theta_0, \theta_1$
will be generated by an adapted framing along a Hopf hypersurface $M \subset X$, provided that
$\Sigma$ satisfies the independence condition
\begin{equation}\label{indcond}\w^1 \& \w^2\& \ldots \& \w^{2n-1} \restr_\Sigma \ne 0.
\end{equation}

Let $\I$ be the Pfaffian system on $\F$ generated (as a differential ideal) by $\theta_0$ and $\theta_1$.
We will later develop methods for constructing integral submanifolds of $\I$ (when $X=\CH^n$,
for small values of $\alpha$) which are based on a close examination of the structure of $\I$
as an algebraic ideal.  Thus, we need the algebraic generators of $\I$.  We compute
\begin{align*}
d\theta_0 &\equiv -\w^{2n}_j \& \w^j,\\
d\theta_1 &\equiv \w^{2n}_j \& \w^{2n}_k J^k_j + \alpha \w^{2n}_k J^k_j \&\w^j + \Phi^{2n}_{2n-1},
\end{align*}
modulo $\theta_0, \theta_1$, where $1\le j,k \le 2n-2$.
(In simplifying the second equation, we use the relation
$\w^{2n-1}_j = \w^{2n}_k J^k_j$, which is a consequence of \eqref{connrel}.)
For convenience, introduce the notation
$$\psi_j := \w^{2n}_j, \qquad 1\le j \le 2n-2.$$
Then, using these abbreviations and the formula for the curvature 2-forms,
 the  generator 2-forms of $\I$ can be expressed as
\begin{align*}\Theta_0 &:= -\psi_j \& \w^j,\\
\Theta_1 &:= (\psi_j - \alpha \w^j) \& \psi_k J^k_j - 2c\, \Omega
\end{align*}
Note that, modulo the 1-forms of $\I$, $\Omega$ is congruent to
$$\Omega' := \w^1 \& \w^2 + \ldots + \w^{2n-3} \& \w^{2n-2}.$$

One can show that, for any $n\ge 2$, the exterior differential system $\I$ is
involutive, with nonzero Cartan characters $s_1=s_2=\ldots=s_{n-1}=2$
(see, e.g., \cite{cfb} for definitions).
Thus, Hopf hypersurfaces in $X$ depend locally on specifying 2 real functions
of $n-1$ real variables.  Our aim in the next section is to give a concrete
(and global) realization of this abstract count, in the hyperbolic case.

\section{Hopf Hypersurfaces with Real Characteristics}\label{hypersurfaces}
From now on assume that $X=\CH^n$ and let $c=-1/r^2$.  Moreover, assume
that $|\alpha| < 2/r$ and let $\phi \in (-\pi/2, \pi/2)$ be the unique
angle such that
$$\alpha = (2/r) \sin\phi.$$
Our aim in this section is to define a pair of characteristic foliations for
integral submanifolds of $\I$, and to show how to construct
the Hopf hypersurface $M$ in terms of data associated to the
characteristics.

Notice that the above 2-forms $\Theta_0,\Theta_1$ are each expressed as a sum of $2n-2$ wedge products of pairs of 1-forms.
By linearly combining $\Theta_0$ and $\Theta_1$, we can produce equivalent generators which are much lower in rank, namely
\begin{multline*}\Theta_1 \pm \tfrac2r \cos\phi\, \Theta_0 = \\
2\sum_{i=1}^{n-1} (\psi_{2i-1} - \tfrac1r (\sin\phi\, \w^{2i-1} \pm \cos\phi\, \w^{2i})) \&
(\psi_{2i} -\tfrac1r(\sin\phi\, \w^{2i} \mp \cos\phi\, \w^{2i-1})).
\end{multline*}
The factors in these $n-1$ wedge products will have special significance in what follows, so we will introduce the
notation
$$\kappa^\pm_j := \psi_j - \tfrac1r(\sin\phi\, \w^j \pm \cos\phi\, \w^\jbar),\qquad 1\le j \le 2n-2,$$
where
$$ \w^\jbar = \left\{ \begin{aligned} \w^{j+1}, &\text{ for $j$ odd}, \\ -\w^{j-1} &\text{ for $j$ even.}
\end{aligned}\right.
$$
Then we may take $\displaystyle\sum_{i=1}^{n-1} \kappa^+_{2i-1} \& \kappa^+_{2i}$ and
$\displaystyle\sum_{i=1}^{n-1} \kappa^-_{2i-1} \& \kappa^-_{2i}$ as generator 2-forms for $\I$, in place of $\Theta_0$ and $\Theta_1$.

We define two Pfaffian systems of rank $2n$ on $\F$, given by
$$\M^\pm = \{ \theta_0,\theta_1, \kappa^\pm_1, \ldots, \kappa^\pm_{2n-2}\}.$$
These are the {\em characteristic systems} of $\I$, so called because
a $2n-2$-dimensional integral hyperplane $E \subset T_u \F$ belongs to the characteristic variety
of $\I$ (see \S4.6 in \cite{cfb})
if and only if either the 1-forms $\kappa_j^+$ or the 1-forms $\kappa_j^-$
restrict to be linearly dependent on $E$.  (If $|\alpha|>2/r$, or $X=\CP^n$, the
complexified characteristic variety of $\I$ has no real points, and thus the system $\I$ is elliptic.)

Within each characteristic system, there is a smaller system which
is {\em completely integrable}, i.e., it satisfies the Frobenius integrability
condition.  If we define
$$\kappa^\pm_0 := \theta_1 \mp \tfrac2r \cos\phi\, \theta_0 = \w^{2n}_{2n-1} - \tfrac2r(
\sin\phi\w^{2n-1}\pm\cos\phi\w^{2n}),$$
then the systems
$$\K^\pm = \{\kappa^\pm_0, \kappa^\pm_1, \ldots, \kappa^\pm_{2n-2} \}$$
are each completely integrable Pfaffian systems on $\F$.
Indeed,
\begin{equation}\label{Dcontact}
d\kappa^\pm_0 \equiv 2\sum_{i=1}^{n-1} \kappa^\pm_{2i-1} \& \kappa^\pm_{2i}\quad \mod \kappa^\pm_0,
\end{equation}
while
$$d\kappa^+_j = -\kappa^\pm_k \& \w^k_j + \w^{2n}_\jbar \& \kappa^\pm_0
+\tfrac1r(\sin\phi\, \w^{2n} \mp\cos\phi\, \w^{2n-1})\& \kappa^\pm_j -(\sin\phi\,\w^{2n-1}\pm\cos\phi\,\w^{2n})\& \kappa^\pm_\jbar.$$
The systems $\K^+$ and $\K^-$ have rank $2n-1$; thus, applying the Frobenius Theorem shows that $\F$ is foliated by codimension-$(2n-1)$ integrals of each system.
Let $L^+$ and $L^-$ be the leaf spaces for each foliation and let $q^\pm : \F \to L^\pm$
take a point in $\F$ to the unique maximal leaf through that point.  Then we have the following result,
to be proven in \S\ref{putoff}:

\begin{prop}\label{NisSphere} The leaf spaces $L^\pm$ can be identified with the standard
sphere $S^{2n-1}$ in a way such that the maps $q^\pm:\F \to S^{2n-1}$ are smooth submersions.
\end{prop}

Now let $\Sigma^{2n-1}\subset \F$ be an integral manifold of $\I$ satisfying the independence condition.
Then clearly the 1-forms $\kappa^\pm_0$ pull back to be zero on $\Sigma$, as do the 2-forms
given by the right-hand side of \eqref{Dcontact}.  This implies that the characteristic
systems $\K^\pm$ each restrict to $\Sigma$ to be of rank at most $n-1$.

\begin{lemma}The restrictions of the system $\K^\pm$ to $\Sigma$ have rank exactly $n-1$ at each point.
\end{lemma}
\begin{proof}Suppose that $\kappa^\pm = T^\pm \omega$ when restricted to $\Sigma$, where we define
vector-valued 1-forms $\omega = (\w^1,\ldots, \w^{2n-2})$
and $\kappa^\pm$ similarly.  Then the vanishing of $\Theta_0$ on $\Sigma$ implies that $T^+$ and $T^-$
are transposes of one another.  However, because the span on $\F$ of $\{\kappa^+_j, \kappa^-_j\}$
is the same as $\{ \psi_j, \omega^j\}$, the ranks of $T^+$ and $T^-$ must add up to $2n-2$.
Hence, each has rank $n-1$.
\end{proof}

It follows that $\Sigma$ is foliated in two ways, by $n$-dimensional
integral manifolds of $\K^+$ and $\K^-$, and the restrictions to $\Sigma$ of the
maps $q^+$ and $q^-$ each have rank $n-1$.

The equation \eqref{Dcontact} implies that on each leaf space $L^\pm$ there is a contact distribution
whose annihilator 1-forms pull back via $q^\pm$ to be multiples of $\kappa^\pm_0$.
(For a general
criterion for the existence of well-defined exterior differential systems on quotient
manifolds, given in terms of semibasic 1-forms on the total space, see Prop. 6.1.19 in \cite{cfb}.)
Hence, the
images of $\Sigma$ under $q^+$ and $q^-$ are $(n-1)$-dimensional immersed contact manifolds in the leaf spaces.
(In \S\ref{putoff} we will give a geometric interpretation of the contact condition in these spaces.)
It is this `data', the images $q^\pm(\Sigma)$, which we will use to recover the Hopf hypersurface.
For simplicity, we consider here only the case where the contact manifolds are embedded.

\begin{theorem}\label{constructHopf}\label{mainthm}
  Let $N_1 \subset L^+$ and $N_2 \subset L^-$ be Legendrian submanifolds of dimension $n-1$,
and let $R \subset \F$ be the intersection of their inverse images under $q^+$ and $q^-$ respectively.
Then $R$ is a smooth $n^2$-dimensional integral manifold of $\I$.  Furthermore,
if $U$ is an open subset of $R$ on which the independence condition holds, then $\rho(U)\subset \CH^n$
is an immersed Hopf hypersurface with Hopf principal curvature $\alpha$.
\end{theorem}

\begin{proof}  First, we note that the following 1-forms comprise a coframe on $\F$:
$$\w^1, \ldots, \w^{2n-1},\theta_0, \theta_1, \psi_1, \ldots, \psi_{2n-2}, \w^{2i}_j,$$
where $1\le i \le n-1$ and $1\le j < 2i$.  However, a coframe
better-adapted to the two foliations of $\F$ is given by
$$\kappa^+_0, \kappa^-_0, \ldots, \kappa^+_{2n-2}, \kappa^-_{2n-2}, \w^{2n-1}, \w^{2i}_j.$$

Since $N_1 \subset L^+$ is a smooth codimension-$n$ submanifold, $P = (q^+)^{-1}(N_1)$ is
a smooth codimension-$n$ submanifold of $\F$.  Moreover, because the generators
of $\K^+$ span the semibasic forms for $q^+$,
these 1-forms satisfy $n$ homogeneous linear equations
when pulled back to $P$.  (Of course, one of these equations is that $\kappa^+_0=0$.)
Therefore, the generators $\kappa^-_0, \kappa^-_j$ of $\K^-$ pull back to
be linearly independent on $P$.  Because these 1-forms span the semibasic forms for $q^-$,
it follows that $q^-$ restricts to $P$ to be a surjective submersion. (In fact, the
image under $q^-$ of any maximal leaf of $\K^+$ is all of $L^-$.)
Hence $R=P \cap (q^-)^{-1}(N_2)$
is a smooth integral manifold of $\I$, of codimension $2n$ in $\F$.

Assume that the independence condition is satisfied on $U\subset R$.
Because the semibasic forms for $\rho$ are spanned by $\w^1, \ldots, \w^{2n}$
and $\w^{2n}=\theta_0$ pulls back to be zero on $R$, then the restriction
of $\rho$ to $U$ has rank $2n-1$.  It follows that $\rho(U)$ is an immersed hypersurface in $\CH^n$.
It remains to show that this hypersurface is Hopf.

Given any $u \subset U$, $\rho$ is locally a submersion
onto a embedded hypersurface $M$ containing $x=\rho(u)$.
Let $\Sigma \subset U$ be a $(2n-1)$-dimensional submanifold containing $u$, and satisfying
the independence condition.  Then $\Sigma$ is the image of an adapted unitary frame field
$(e_1,\ldots, e_{2n})$ defined on $M$.  Because $\theta_0$ and
$\theta_1$ pull back to be zero on $\Sigma$, it follows that $e_{2n}$ is normal to the hypersurface,
and the structure vector $e_{2n-1}$ is principal with principal curvature $\alpha$, respectively.
\end{proof}

\begin{prop}  Let $R$ be as in Theorem \ref{constructHopf}.
Then the rank of the restriction of $\rho$ to $R$ is at least $n$ at each point.
\end{prop}
\begin{proof}
On $R$ the 1-forms $\kappa^+_j$ satisfy $n-1$ independent homogeneous
linear equations, and the same is true for the $\kappa^-_j$.
Write these linear relations in the form
$$ A \kappa^+ =0, \qquad B\kappa^-=0,$$
where $A,B$ are $(n-1) \times (2n-2)$ matrices of rank $n-1$ whose entries are functions on $R$,
and we group the forms $\kappa^\pm_j$ into vector-valued 1-forms $\kappa^+, \kappa^-$
with $2n-2$ components each.
Using the formulas for the $\kappa^\pm_j$ in terms of
the canonical and connection forms of $\F$, the above relations imply that
\begin{equation}\label{omegaq}
\begin{pmatrix}
A & \tfrac{\cos\phi}{r} A \\ B & -\tfrac{\cos\phi}{r} B \end{pmatrix}
\begin{bmatrix} \psi - \tfrac{\sin\phi}{r} \w \\ \tilde J\w \end{bmatrix} = 0,
\end{equation}
where $\psi$ has components $\psi_j$, $\w$ has components $\w^j$ for $1\le j \le 2n-2$,
and $\tilde J$ is the submatrix of $J$ obtained by deleting the last two rows and columns.
Any linear relation that holds among the pullbacks of the $\w^j$ to $R$ arises by linearly combining the rows
of the matrix on the left in \eqref{omegaq} to obtain a row with the first
$2n-2$ entries equal to zero.
If we let $k$ be the rank of the square matrix $\left(\begin{smallmatrix}A \\ B\end{smallmatrix}\right)$,
then the number of linear relations among the $\w^j$ is at most $2n-2-k$.
It follows that the number of linear independent
1-forms among $\w^1,\ldots, \w^{2n-1}$ restricted to $R$ is at least $k+1$, which is
at least $n$.
\end{proof}
Of course, for generic pairs of contact submanifolds $N_1, N_2$ we
expect the rank of $\rho$ to be equal to $2n-1$ on an open dense set in $R$.
For the case $n=2$, we will
express this genericity condition in concrete form  at the end of \S4.

\begin{prop} Let $N_1, N_2$ and $R$ be as in Theorem \ref{mainthm}.
Assume in addition that $N_1$ and $N_2$ are closed, and the independence
condition holds on all of $R$.  Then $M =\rho(R)$ is a complete
Hopf hypersurface in $\CH^n$.
\end{prop}
\begin{proof}  Since $R$ is the intersection of closed sets $(q^+)^{-1}(N_1)$ and
$(q^-)^{-1}(N_2)$, then $R$ is closed.   We will show that $M$ is closed.

Let $p_k$ be a sequence of points in $M$ converging to $\pbar \in \CH^n$.
Let $\Bbar$ be a geodesic ball centered at $\pbar$, such that $\F$ is trivial
on an open set containing $\Bbar$.  Then $\rho^{-1}(\Bbar)$ is a compact
set, diffeomorphic to $\Bbar \times U(n)$.
Choose a sequence $q_k \in R$ such that $\rho(q_k)=p_k$.  Then for
$k$ sufficiently large, the points $q_k$ lie in $\rho^{-1}(\Bbar)$.
Thus, there is a subsequence converging to a point $\qbar \in R$.
Then by continuity of $\rho$, $\pbar = \rho(\qbar)$ belongs in $M$.

Because $M$ is a closed submanifold of the complete space $\CH^n$,  $M$ is also complete.
\end{proof}

\section{Geometry of the Leaf Spaces}
\subsection{The Hopf Fibration and Moving Frames}\label{Qstructure}
In this section we review the geometry that $\CP^n$ and $\CH^n$ inherit via
the Hopf fibration.  We will later use this fibration to establish
results about the leaf spaces discussed in \S\ref{hypersurfaces},
and to construct concrete examples of Hopf hypersurfaces in $\CH^n$.

On $\C^{n+1}$ define the real quadratic form
$$\langle \vz,\vz \rangle= |z_1|^2 + \cdots +|z_n|^2 + \emp |z_0|^2, \qquad \emp = \pm 1,$$
and let $Q \subset \C^{n+1}$ be the quadric hypersurface defined by
$\langle \vz,\vz\rangle=\emp r^2$ for $r>0$.
Let $Q$ have the metric given by restricting this quadratic form; this metric will be Riemannian
if $\emp=1$ and Lorentzian if $\emp=-1$.
The {\em Hopf fibration} $\pi : Q \to X$ is the quotient
by the isometric action $\vz \mapsto e^{\ri \theta} \vz$.
of $S^1$ on $Q$.  When $\emp = +1$, $X$ is $\CP^n(r)$, and when $\emp=-1$,
$X$ is $\CH^n(r)$.  We endow $X$ with the unique Riemannian metric such that the Hopf fibration
is a (semi)-Riemannian submersion.  The complex structure on $X$ is also inherited
via this submersion.

We say that an orthogonal basis $\ehat_0, \ehat_1,\ldots, \ehat_{2n}$ for $T_\vz Q$ is an {\em adapted unitary frame}
at $\vz$ if, as vectors in $\C^{n+1}$, the members of the frame satisfy
\begin{equation}\label{ezero}
\ehat_0 = \dfrac{\ri}{r} \vz , \quad \ehat_2 = \ri \ehat_1, \quad \ehat_4 = \ri \ehat_3, \ \ldots
\end{equation}
and $\langle \ehat_j, \ehat_j \rangle = 1$ for $1\le j \le 2n$.
(Note that $\langle \ehat_0, \ehat_0 \rangle = \emp$.)
Let $\G$ be the sub-bundle of the general linear frame bundle of $Q$ whose fiber at
$\vz$ consists of all adapted unitary frames for $T_{\vz}Q$.  Then the structure group of $\G$ is
$U(n)$.  A key observation, which we will make use of below in proving Proposition \ref{NisSphere},
is that $\G$ can be identified with a matrix Lie group, by taking the
vectors $\frac{1}{r} \vz, \ehat_1, \ehat_3, \ldots, \ehat_{2n-1}$ as columns
in an $(n+1)\times (n+1)$ complex-valued matrix.  Thus $\G$ is identified
with $U(n+1)$ when $\emp=1$ or with $U(1,n)$ when $\emp=-1$.

We define a submersion $\Pi:\G \to \F$ that sends
an adapted unitary frame $(\ehat_0,\ehat_1, \ehat_2, \ldots, \ehat_{2n})$ at $\vz$
to the unitary frame $(\pi_* \ehat_1, \pi_* \ehat_2, \ldots, \pi_* \ehat_{2n})$ at $\pi(\vz)$.
Conversely, given a unitary frame
$(e_1,\ldots ,e_{2n})$ at $x \in X$ and a point $\vz \in Q$ such that
$\pi(\vz)=x$, there is a unique adapted unitary frame at $\vz$ such that
$\pi_* \ehat_j = e_j$; we will refer to this as the {\em horizontal lift} at $\vz$
of $(e_1, \ldots, e_{2n})$.
Similarly, given a vector field $V$ on an open set $\setU\subset X$, we
may define a vector field $V^H$ on $\pi^{-1}(\setU) \subset Q$ such that
$V^H$ is orthogonal to the fibers of $\pi$, and $\pi_* (V^H)_\vz = V_{\pi(\vz)}$ for each $\vz\in \pi^{-1}(\setU)$; this $V^H$ is called the {\em horizonal lift} of $V$.

On $\G$, we regard $\vz$ (the basepoint map) and $\ehat_\alpha$ as $\C^{n+1}$-valued
functions.  We define 1-forms $\eta^\alpha$ on $\G$
such that
\begin{equation}\label{deez}
d\vz = \ehat_\alpha \eta^\alpha.
\end{equation}
(In what follows, we let $0 \le \alpha,\beta \le 2n$ and $1\le j,k,\ell,m \le 2n$.)
Multiplying by $\ri/r$ gives
$$d\ehat_0 = -\frac1{r^2} \vz \,\eta^0 + \frac1r \ehat_j J^j_k \eta^k,$$
where $J$ is defined by \eqref{defofJmatrix}.
We also define 1-forms $\eta^j_k$ on $\G$ by resolving
the derivatives of the frame vectors:
\begin{equation}\label{deen}
d\ehat_j = \ehat_k \eta^k_j - \frac{\emp}{r^2} \vz \,\eta^j
- \frac{\emp}r \ehat_0 \,J^j_k \eta^k.
\end{equation}
(The last two terms are deduced by differentiating $\langle \vz, \ehat_j \rangle = 0$ and $\langle \ehat_0, \ehat_j \rangle =0$.)
Computing $d\langle \ehat_j, \ehat_k \rangle =0$ shows that
the $\eta^j_k$ are skew-symmetric, while differentiating \eqref{ezero} gives
\begin{equation}\label{commutewithJ}
J^k_\ell \eta^\ell_j = \eta^k_\ell J^\ell_j.
\end{equation}
We note that there are no further relations among these 1-forms;
in fact, the forms $\eta^\alpha$ and $\eta^j_k$ (for, say, $j$ even and $k < j$)
together comprise a coframe on $\G$.  In fact, these are a basis for the left-invariant
1-forms when $\G$ is regarded as a Lie group.

The structure equations for the coframe on $\G$ are obtained by
taking exterior derivatives of the defining equations for these 1-forms.
Differentiating \eqref{deez} gives the structure equations
\begin{equation}\label{strone}
d\begin{bmatrix}\eta^0 \\ \eta^j \end{bmatrix}
=
- \begin{pmatrix} 0 & -\frac{\emp}{r} J^k_j \eta^j \\
\frac{1}{r} J^j_k \eta^k & \eta^j_k \end{pmatrix} \&
\begin{bmatrix} \eta^0 \\ \eta^k \end{bmatrix}.
\end{equation}
while differentiating \eqref{deen} leads to
\begin{equation}\label{strtwo}
\begin{aligned} d\eta^j_k
&= -\eta^j_\ell \& \eta^\ell_k +\dfrac{\emp}{r^2}
\left(J^j_\ell \eta^\ell \& J^k_m \eta^m + \eta^j \& \eta^k\right).
\end{aligned}
\end{equation}
Next, we need to calculate the pullbacks under $\Pi$ of the canonical and connection 1-forms on $\F$.

\begin{lemma}\label{pullbackvalues}
 $\Pi^* \w^j = \eta^j$ and $\Pi^*\w^j_k = \eta^j_k - \tfrac1{r} J^j_k \eta^0.$
\end{lemma}
\begin{proof}
Let $\fhat = (\ehat_0, \ldots, \ehat_{2n})$ be an arbitrary section of $\G$, defined on an open set
$\Uhat \subset Q$.  Then \eqref{deez} implies that the 1-forms $\fhat^* \eta^\alpha$
are dual to the frame vectors $\ehat_\alpha$, i.e.,
\begin{equation}\label{vdecompose}
\bv = \left( \bv \intprod \fhat^* \eta^\alpha\right) \ehat_\alpha, \qquad \forall \bv \in T_\vz Q, \ \vz \in \Uhat.
\end{equation}
Hence, the 1-forms $\eta^j$ on $\G$ are the pullbacks of the canonical 1-forms on the
general frame bundle of $Q$.

Now let $f=(e_1, \ldots, e_{2n})$ be an arbitrary section of $\F$, defined on an open set $\setU \subset X$,
and let $\sigma :\setU \to Q$ be an arbitrary lift (i.e., $\pi \circ \sigma$ is the identity on $\setU$).
Let $\ehat_j$ be horizontal lifts of the $e_j$, and let $\ehat_0$ be such that $(\ehat_0, \ldots,\ehat_{2n})$
is an adapted unitary frame on $\sigma(\setU)$.  Use pushforward by the $S^1$ action on $Q$ (which
preserves horizontality) to extend this frame to a section $\fhat$ defined on an
open set $\Uhat$ containing $\sigma(\setU)$.  These maps are summarized by
the following commutative diagram.
\setlength{\unitlength}{1mm}
\begin{center}
\begin{picture}(40,25)(0,5)
\put(0,5){$\CH^n \supset \setU$}
\put(4,25){$Q \supset \Uhat$}
\put(13,10){\vector(0,1){13}}
\put(10,16){$\sigma$}
\put(16,23){\vector(0,-1){13}}
\put(17,16){$\pi$}
\put(17,7){\vector(1,0){15}}
\put(23,10){$f$}
\put(17,26){\vector(1,0){15}}
\put(23,29){$\fhat$}
\put(33,5){$\F$}
\put(33,25){$\G$}
\put(34,23){\vector(0,-1){13}}
\put(36,16){$\Pi$}
\end{picture}
\end{center}

For any vector $\bv \in T_\vz Q$, $\vz \in \Uhat$, we can decompose $\bv$ as in \eqref{vdecompose}.
By applying $\pi_*$ to each side, we conclude that
$$\bv \intprod \fhat^* \eta^j = (\pi_* \bv) \intprod f^* \w^j.$$
Using the commutative diagram, we obtain
$$\bv \intprod \fhat^* \eta^j =\bv \intprod (f\circ \pi )^*\w^j = \bv \intprod (\Pi \circ \fhat)^* \w^j,$$
and conclude that $\fhat ^*(\eta^j - \Pi^* \w^j)=0$.  Since $f$ and $\sigma$ are arbitrary,
then $\Pi^* \w^j = \eta^j$.

Similarly, because the matrix in \eqref{strone} gives the connection forms for the Levi-Civita connection
$\widehat\nabla$ on $Q$, we have
$$\widehat\nabla_\bv \ehat_j = (\bv \intprod J^j_k \fhat^*\eta^k) \ehat_0
+ (\bv \intprod \fhat^* \eta^k_j) \ehat_k$$
for any section $\fhat$ of $\G$.  Then applying $\pi_*$ to both sides,
and using the Riemannian submersion property
that $\pi_* \left( \widehat{\nabla}_{V^H}\,{W^H}\right) = \nabla_V W$
for any vector fields $V,W$ on $X$, we conclude that
$$\Pi^*\w^j_k = \eta^j_k - \tfrac{1}{r} J^j_k \eta^0.$$
\vskip -.2in
\end{proof}

We note that by pulling back equation \eqref{defofCurvature} via $\Pi$ and using
the results of Lemma \ref{pullbackvalues} we can calculate the curvature forms
given by \eqref{curvaturevalues}.

\subsection{Leaf Spaces and Spheres}\label{putoff}
\begin{proof}[Proof of Proposition \ref{NisSphere}]
We now assume that $\emp=-1$ and let $1\le j,k\le 2n-2$.
Let $\Khat^\pm$ be the Pfaffian systems on $\G$ spanned
by the pullbacks, via $\Pi$, of the characteristic systems $\K^\pm$ on $\F$, i.e.,
$\Khat^\pm$ is spanned by the 1-forms
\begin{align*}
\khat^\pm_0 := \Pi^* \kappa^\pm_0 &=\eta^{2n}_{2n-1} - \dfrac{1}{r}\eta^0-\dfrac{2}{r}(\sin\phi\, \eta^{2n-1}\pm\cos\phi\,\eta^{2n})
\intertext{and}
\khat^\pm_j := \Pi^* \kappa^\pm_j &= \eta^{2n}_j - \dfrac{1}{r}(\sin\phi\, \eta^j \pm \cos\phi \eta^\jbar).
\end{align*}
The Pfaffian systems $\Khat^\pm$ satisfy the Frobenius condition, and maximal integral manifolds of $\Khat^\pm$
are in 1-to-1 correspondence, via $\Pi$, with those of $\K^\pm$.  However, the 1-forms $\khat^\pm_0$ and $\khat^\pm_j$
are left-invariant 1-forms on $\G$, and thus the maximal integral manifolds of $\Khat^\pm$ are left cosets
of codimension-$(2n-1)$ connected Lie subgroups $\H^\pm \subset \G$.  We give the leaf spaces $L^\pm$ a manifold structure
by identifying them with the homogeneous spaces $\G/ \H^\pm$; then it is automatic that the quotient maps
$q^\pm$ are smooth submersions.

It remains to identify the spaces $L^\pm$ as spheres.
Let $\V$ be the cone of nonzero null vectors (with respect to $\langle,\rangle$ ) in $\C^{n+1}$.  The set $\P\V$ of
complex null lines in $\C^{n+1}$ is the image of $\V$ under the
projectivization map $\pi$.  We identify $\P\V$ with the unit sphere $S^{2n-1}\subset \C^n$ since,
in terms of homogeneous coordinates $[z_0, z_1, \ldots, z_n]$ on $\CP^n$,
$\P\V$ lies entirely in the domain of the affine coordinates defined by
$w_1=z_1/z_0$, \ldots, $w_n=z_n/z_0$, and the image of $\pi(\V)$ is the
codimension-one set in $\C^n$ defined by $|w_1|^2 + \cdots + |w_n|^2=1$.

Next, define maps $\bg^\pm:\G \to\V$ by
\begin{equation}\label{biggy}
\bg^\pm = \ehat_0 - (\sin\phi\, \ehat_{2n-1} \pm \cos\phi\, \ehat_{2n}).
\end{equation}
Let $\bg^\pm_\C = \pi \circ \bg^\pm$ be the corresponding maps from $\G$ to $S^{2n-1}$.
Because any null vector can be expressed as the sum of spacelike and timelike vectors, the maps $\bg^\pm_\C$
are surjective.  Moreover, the fibers of $\bg^+_\C$ (respectively, $\bg^-_\C$) are cosets of $\H^+$ (resp., $\H^-$).
To see why, first note that each fiber of $\bg^+_\C$ is acted on simply transitively by the isotropy subgroup of $\G$ preserving
a null line in $\C^{n+1}$; in particular, such isotropy subgroups are connected.  Then, differentiating
$\bg^+$ shows that $\bg^+_\C$ is constant along the cosets of $\H^+$:
\begin{multline}\label{dgplus}
d\bg^+ = \frac1r \bg^+ (\sin\phi\, \eta^{2n} - \cos \phi\, \eta^{2n-1})
+\frac{\ri}r \bg^+ (\eta^0 - \sin \phi\, \eta^{2n-1} - \cos\phi\, \eta^{2n}) \\
+ (\cos\phi \, \ehat_{2n-1} - \sin\phi\,\ehat_{2n})\khat^+_0 +(\cos\phi\, \ehat_j -\sin \phi\,\ehat_\jbar) \khat^+_j.
\end{multline}
(This follows by noticing that the first two terms are complex multiples of $\bg^+$ while all remaining
terms involve 1-forms in $\Khat^+$.)   Thus, each coset lies in a single fiber of $\bg^+_\C$,
but since that fiber is connected, the fiber consists of exactly one coset of $\H^+$.
Therefore, $\bg^+$ covers a bijective map $\gamma^+ : L^+ \to \P\V \cong S^{2n-1}$.  We summarize this situation in the following commutative diagram.
\setlength{\unitlength}{1mm}
\begin{center}
\begin{picture}(40,30)(0,10)
\put(5,32){$\G$}
\put(10,33){\vector(1,0){17}}
\put(17,34.5){$\bg^+$}
\put(30,32){$\mathcal V$}
\put(31,30){\vector(0,-1){15}}
\put(33,22){$\pi$}
\put(11,30){\vector(1,-1){16}}
\put(17,25){$\bg^+_\C$}
\put(-5,22){$q^+\!\circ \Pi$}
\put(7,30){\vector(0,-1){15}}
\put(4,9){$L^+$}
\put(11,11){\vector(1,0){16}}
\put(16,12.5){$\gamma^+$}
\put(29,9){$S^{2n-1}$}
\end{picture}
\end{center}
In addition, because the semibasic 1-forms $\khat^+_0, \ldots, \khat^+_{2n-2}$ appear in \eqref{dgplus} with
linearly independent vector coefficients, the map $\gamma^+:L^+ \to S^{2n-1}$ is a local diffeomorphism.
Hence, $L^+$ is diffeomorphic to the sphere, and the same is true for $L^-$.
\end{proof}

Once we identify the leaf spaces $L^+$ and $L^-$ with the projectivized null cone, the
contact structures on these spaces have a geometric interpretation.
\begin{prop}\label{contactcond}
Let  $N \subset L^\pm$ be a submanifold, and
let $\bn:N \to \mathcal V$ be any lift of $\gamma^\pm \restr_N$.
Then $N$ is a Legendrian submanifold if and only if
\begin{equation}\label{nullcond}
\left\langle d\bn, \ri \bn \right\rangle=0.
\end{equation}
\end{prop}
Note that the left-hand side of \eqref{nullcond} is a real-valued differential 1-form on $N$, and
that this condition is independent of the choice of lift $\bn$.
In fact, when we identify the projectivized null cone with $S^{2n-1}$, this proposition
implies that the contact $(n-1)$-planes are perpendicular to the fibers of the Hopf fibration.

\begin{proof}
Again, we prove the assertions for $L^+$, the proof for $L^-$ being similar.
Let $\hat{q}=q^+ \circ \Pi: \G \to L^+$.
By definition of the contact structure on $L^+$, $N$ is Legendrian
if and only if $\hat{q}^{-1}(N)$ is an integral of the 1-form $\khat^+_0 = \Pi^* \kappa^+_0$.

Let $f:N\to \G$ be any lift of $\bn:N \to \mathcal V$, so that $\bn = \bg^+ \circ f$.
Then pulling \eqref{dgplus} back by $f$ gives
\begin{multline}\label{uglyv}
d\bn = \dfrac1r \bn f^*(\sin\phi\,  \eta^{2n}-\cos\phi\, \eta^{2n-1})
+ \dfrac{\ri}r \bn f^*(\eta^0 - \sin \phi\, \eta^{2n-1} - \cos\phi\, \eta^{2n}) \\
+(\cos\phi \, \ehat_{2n-1} - \sin\phi\,\ehat_{2n})f^*\khat^+_0
+(\cos\phi\, \ehat_j -\sin \phi\,\ehat_\jbar) f^*\khat^+_j,
\end{multline}
where the vector-valued functions $\ehat_j$ on $\G$ should be understood as being composed with $f$.
By using the formula for $\bg^+$ we obtain
$$\langle d\bn, \ri \bn \rangle = f^* \khat^+_0.$$
Since $f$ is arbitrary, the vanishing of the left-hand side is equivalent
to $\hat{q}^{-1}(N)$ being an integral of $\khat^+_0$.
\end{proof}

\section{Parametric Representations}
Given two Legendrian submanifolds $N_1, N_2$ in $S^{2n-1}$, we 
wish to construct the Hopf hypersurface $M$ of Theorem \ref{constructHopf}.
Let $R$ be as in Theorem \ref{constructHopf}
and let $\hat{R}$ be its inverse image under $\Pi: \G \to \F$.
This is the intersection of the inverse images of $N_1, N_2$ under the maps
$\bg^+$ and $\bg^-$.  The Hopf hypersurface is the image of $\hat{R}$ under $\rho \circ \Pi$, the map that takes a point
in the frame bundle $\G$ to the projectivization of the basepoint $\vz$.

Let $\bn_i: N_i \to \V$ be any lifts of $N_i$ into the null cone.  (We may choose unique lifts
by requiring that the images lie in the intersection of $\V$ with the hyperplane where $z_0=1$.)
We will determine the values of $\bg^+, \bg^-, \ehat_0$ and $\bz$ along $\hat{R}$,
in terms of $\bn_1, \bn_2$ and an extra parameter $\lambda$.
Because $\bg^+$ (resp. $\bg^-$) must lie in the complex line spanned by $\bn^+$
(resp. $\bn^-$), there are nonzero scalars $\mu_1, \mu_2$ such that
\begin{equation}\label{gforms}
\bg^+ = \mu_1 \bn_1, \qquad
\bg^- = \mu_2 \bn_2
\end{equation}
These scalars are not arbitrary, since the intersection of a
fiber of $\hat{q}^+$ with a fiber of $\hat{q}^-$ is a 3-dimensional submanifold of $\G$.
In fact, each such submanifold intersects the fiber of $\rho \circ \Pi$
in a 2-dimensional torus.\footnote{This is the orbit of an $S^1 \times S^1$ action on
 $\G$.  One copy of $S^1$ acts by rotating $\ehat_1$ and $\ehat_2$ in a circle,
while the other acts by multiplying $\vz$ by a unit modulus complex number at the
same time as rotating $\ehat_{2n-1}$ and $\ehat_{2n}$ in a circle.}
Thus, only one additional real parameter, independent of local coordinates on $N_1, N_2$, should remain on
the image of $\hat{R}$ under $\rho \circ \Pi$.

The multiples $\mu_1, \mu_2$ are partially determined by the normalizations of the vectors $\bg^\pm$. For example,
\eqref{biggy} implies that $\bg^+ + \bg^- = 2(\ehat_0 - \sin\phi\, \ehat_{2n-1})$, so that
$\langle \bg^+ + \bg^- , \bg^+ + \bg^- \rangle = -4\cos^2\phi.$
Substituting \eqref{gforms} into equation gives
\begin{equation}\label{profirstcond}
\langle \mu_1 \bn_1, \mu_2 \bn_2\rangle = -2\cos^2 \phi.
\end{equation}
To facilitate factoring out the scalars, we introduce the sesquilinear form
\begin{equation}\label{mydot}
\langle \vz, \vw\rangle_\C = -\overline{z_0} w_0 + \overline{\mu_1} w_1 +\ldots+ \overline{z_n} w_n,
\end{equation}
which satisfies $\langle \vz, \vw \rangle = \realpart \langle \vz, \vw\rangle_\C$.
Define the function $\zeta:N_1 \times N_2 \to \C$ by
\begin{equation}\label{zprop}
\zeta := \langle \bn_1, \bn_2 \rangle_\C.
\end{equation}
Then \eqref{profirstcond} gives
\begin{equation}\label{firstcond}
 \realpart \left(\overline{\mu_1} \mu_2 \zeta \right) =-2\cos^2\phi.
\end{equation}
Since the right-hand side must be nonzero, we restrict
our attention to the open set $U\subset N_1\times N_2$ where $\zeta\ne 0$.

From \eqref{biggy} one can calculate that
\begin{equation}\label{ehatzero}
\ehat_0 = \dfrac12(1+\ri \tan \phi) \bg^+ + \dfrac12(1-\ri \tan \phi) \bg^-
= \dfrac{e^{\ri\phi}\mu_1}{2\cos\phi} \bn_1 +\dfrac{e^{-\ri\phi}\mu_2}{2\cos\phi}\bn_2.
\end{equation}
Then the condition $\langle \ehat_0, \ehat_0 \rangle = -1$ implies that
$$\realpart\left( e^{-2\ri \phi}\overline{\mu_1}\mu_2 \zeta \right)=-2\cos^2\phi.$$
Comparing with \eqref{firstcond}, we deduce that $\overline{\mu_1}\mu_2 \zeta = -2e^{\ri \phi}\cos\phi$.
Hence, the coefficients of $\bn_1, \bn_2$ in \eqref{ehatzero} satisfy
$$\overline{\dfrac{e^{\ri\phi}\mu_1}{2\cos\phi}} \dfrac{e^{-\ri\phi}\mu_2}{2\cos\phi} =
-\dfrac{1-\ri\tan\phi}{2\zeta}.$$
Suppose that there exists a smooth complex-valued function $\tau:N_1 \times N_2 \to \C$ such that
$\tau^2= \tfrac12(1-\ri\tan\phi)/\zeta$.
(Equivalently, assume that a consistent branch can be chosen for $\sqrt{\zeta}$ on $U$.)
Then $e^{\ri\phi} \mu_1 = 2\cos\phi\overline{\lambda \tau}$ and 
$e^{-\ri\phi}\mu_2 = -2\cos\phi \tau/\lambda$ for some nonzero complex parameter $\lambda$,
and so
\begin{equation}\label{ehatvalu}
\ehat_0 = \overline{\lambda\tau} \bn_1 -  \dfrac{\tau}{\lambda} \bn_2.
\end{equation}

Because $\vz \in Q$ is a complex multiple of $\ehat_0$, the image of $\hat R$ under $\rho \circ \Pi:\G \to \CH^n$
coincides with the image of $\ehat_0$ under complex projectivization $\pi$.
Note that this image does not depend on the argument of $\lambda$; thus, the Hopf hypersurface
$M\subset \CH^n$ is given, up to complex multiple, by
the right-hand side of \eqref{ehatvalu}, and $M$ is parametrized by local coordinates on $U$ and $|\lambda|$.
So, we may take $\lambda$ to be real and positive; then using $\vz = -(\ri/r) \ehat_0$ gives the following
mapping from $U \times \R^*$ into $\C^{n+1}$ which is a canonical lift of the Hopf hypersurface in to $Q$:
\begin{equation}\label{dalembert}
\vz = \dfrac1{r}\left( -\ri \lambda\overline{\tau} \bn_1 + \ri \dfrac{\tau}{\lambda} \bn_2\right)
\end{equation}

\begin{remark}
One can show that the $\lambda$-coordinate curves on the lift correspond to the $W$-curves on $M$, i.e.,
the lines of curvature tangent to the
structure vector $W$.  Moreover, \eqref{dalembert} shows that when we use affine
coordinates $w_1=z_1/z_0$, $w_2=z_2/z_0, \ldots$ to map $M$ into the open unit ball in $\C^n$,
the Legendrian submanifolds $N_1, N_2$ appear as limits of $M$ on the ideal boundary $S^{2n-1}$ of $\CH^n$ as
$\lambda$ approaches 0 and $\infty$ along the $W$-curves.  It is also clear from \eqref{dalembert}
that each $W$-curve is contained in the complex line in $\CP^n$ spanned by its endpoints on the boundary.
The intersection of $\CH^n$ with this line is a hyperbolic disc in which the $W$-curve has constant curvature $\alpha$.
\end{remark}

Note that for {\em arbitrary} given Legendrian submanifolds $N_1, N_2$, the composition of the mapping
\eqref{dalembert} with projection into $\CH^n$ is not automatically of rank $2n-1$.
We end this section with a discussion of the case $n=2$, in which
we work out the rank condition explicity.

Our starting ingredients are
contact curves $\CC_1, \CC_2$ in $S^3$.  Suppose a curve $\CC$ in $S^3 \subset \C^2$ is
defined parametrically by complex coordinates
$w_1(t) = e^{\ri \beta(t)}\cos \mu(t) $ and $w_2 = e^{\ri \gamma(t)}\sin \mu(t) $,
and $\bn(t) = [1, w_1(t), w_2(t)]$ is the lift into the null cone $\V \subset \C^3$.
Applying the condition from Prop. \ref{contactcond} to this lift, we find that $\CC$ is a contact curve
if and only if
\begin{equation}\label{abccond}
\beta'  \cos^2 \mu + \gamma'  \sin^2 \mu = 0.
\end{equation}
This is an underdetermined ODE which, for example, may be solved for $\beta(t)$ given
functions $\gamma(t)$ and $\mu(t)$.

Let $\CC_1$ and $\CC_2$ be contact curves defined respectively by
solutions $\mu_1(s),\beta_1(s),\gamma_1(s)$ and $\mu_2(t), \beta_2(t), \gamma_2(t)$
of \eqref{abccond}.  The lifts are
$$\bn_1(s) = \begin{bmatrix}1 \\ e^{\ri \beta_1}\cos \mu_1 \\ e^{\ri \gamma_1} \sin \mu_1\end{bmatrix}, \qquad
\bn_2(t) =  \begin{bmatrix}1 \\ e^{\ri \beta_2}\cos \mu_2
\\ e^{\ri \gamma_2} \sin \mu_2\end{bmatrix}.
$$
and we compute
$$\zeta(s,t) = -1+ e^{\ri (\beta_2 - \beta_1)} \cos \mu_1 \cos \mu_2
+e^{\ri (\gamma_2 -\gamma_1)}\sin\mu_1 \sin\mu_2.$$

\begin{prop}\label{geneprop} Let $u = \log \lambda$ for $\lambda$ real and positive.  Then the map taking
$(s,t,u)$ to the projection of the right-hand side of \eqref{dalembert} into $\CH^2$ has rank 3 at points
where $\zeta \ne 0$ and
$$\impart\left( \dfrac{\di^2 \zeta}{\di s \di t} - \dfrac{\sec^2 \phi}{\zeta} \dfrac{\di \zeta}{\di s}\dfrac{\di \zeta}{\di t}
\right) \ne 0.$$
\end{prop}
Twhile the assumption that $\bn_1(s)$ satisfies the contact condition of Prop. \eqref{contactcond}
implies that
$$\left \langle \dfrac{\di \bn_1}{\di s}, \bn_1 \right\rangle_\C = 0.$$

\begin{proof}[Proof of Prop. \ref{geneprop}]  It is sufficient to work
with the projection of the right-hand side of \eqref{ehatvalu}, which differs from $\vz$ by a complex constant.
Differentiating \eqref{ehatvalu} gives
\begin{equation}\label{dezu}
\di \ehat_0/\di u = \lambda\overline{\tau} \bn_1+ (\tau/\lambda) \bn_2.
\end{equation}
It follows from \eqref{zprop} that $\langle \di \ehat_0/\di u, \di \ehat_0/\di u \rangle =1$.  Thus,
$\di \ehat_0/\di u$ is a spacelike vector in $\C^3$, and is linearly independent of
complex multiples of the timelike vector $\ehat_0$.

Next, we compute
$$
\dfrac{\di \ehat_0}{\di s} = \lambda\overline{\tau}_s \bn_1+ \lambda\overline{\tau} \dfrac{d\bn_1}{d s}
-  \dfrac{\tau_s }{\lambda} \bn_2
,\qquad
\dfrac{\di \ehat_0}{\di t} =
\lambda\overline{\tau}_t \bn_1 -  \dfrac{\tau_t }{\lambda} \bn_2 - \dfrac{\tau}{\lambda}\dfrac{d \bn_2}{d t}.
$$
Let $\sfP$ denote projection onto the orthogonal complement of the complex
span of $\ehat_0$ and $\di \ehat_0/\di_u$.  (Note that $\bn_1$ and $\bn_2$ lie in this span,
and the map $\sfP$ is $\C$-linear.)
Then $\sfP(\di \ehat_0/\di s) = \lambda \overline{\tau} \sfP(d \bn_1/d s)$
and $\sfP(\di\ehat_0/\di t) =-(\tau/\lambda) \sfP(d \bn_2/d t)$.
We will compute conditions under which these projections are
linearly independent over $\R$.

The fact that $\bn_1(s)$ satisfies the contact condition of Prop. \eqref{contactcond}
implies
$$\left \langle \dfrac{d \bn_1}{d s}, \bn_1 \right\rangle_\C = 0.$$
Using this and \eqref{ehatvalu}
we obtain
$$
\left\langle \dfrac{d \bn_1}{d s},\ehat_0\right\rangle_\C =-\dfrac{\tau}{\lambda} \dfrac{\di \zeta}{\di s},
\qquad
\left\langle\dfrac{d \bn_1}{d s}, \dfrac{\di \ehat_0}{\di u}\right\rangle_\C =
 \dfrac{\tau}{\lambda}
 \dfrac{\di \zeta}{\di s}.
$$
Hence
$$\sfP\left(\dfrac{d \bn_1}{d s}\right)
= \dfrac{d \bn_1}{d s}
-\dfrac{\overline{\tau}}{\lambda}
 \dfrac{\di \overline{\zeta}}{\di s}
\left( \ehat_0 + \dfrac{\di \ehat_0}{\di u}\right)
= \dfrac{d \bn_1}{d s}
-2\overline\tau^2 \dfrac{\di \overline{\zeta}}{\di s} \bn_1.
$$
Similarly, we compute
$$\sfP\left(\dfrac{d \bn_2}{d t}\right)
= \dfrac{d \bn_2}{d t}
-2\tau^2 \dfrac{\di \zeta}{\di t} \bn_2.
$$
Because the range of $\sfP$ has real dimension 2, 
these two vectors will be linearly independent if and only if
the imaginary part of their complex inner product (in the sense of \eqref{mydot})
is nonzero.  By using the inner product values given by
differentiating \eqref{zprop}, we obtain
$$\left\langle
\dfrac{\di \bn_1}{\di s}
-2\overline{\tau}^2 \dfrac{\di \overline\zeta}{\di s} \bn_1,
\dfrac{\di \bn_2}{\di t}
-2\tau^2 \dfrac{\di \zeta}{\di t} \bn_2
\right\rangle_\C
= \dfrac{\di^2 \zeta}{\di s\di t} - \dfrac{\sec^2 \phi}{\zeta}
\dfrac{\di \zeta}{\di s}\dfrac{\di \zeta}{\di t}.$$
\end{proof}

\section*{Conclusion}
As with the work of Aledo {\it et al} \cite{AGM} on flat surfaces in $S^3$,
one could use the d'Alembert formula to study the Cauchy problem for
Hopf hypersurfaces.  (In this direction, Ryan and I \cite{IR3} have shown that
a Hopf hypersurface may be constructed which contains any given curve with zero holomorphic torsion
in $\CP^2$ or $\CH^2$ as a principal curve.)
Another interesting application of this technique would be the study of singularities
of Hopf hypersurfaces, for example comparing the degenerations of hypersurfaces with small $\alpha$
with those of hypersurfaces with large $\alpha$.

\bigskip
I thank Pat Ryan and Robert Bryant for helpful discussions and encouragement.  I
also thank the referees of earlier versions of this article for their comments and suggestions.

\end{document}